\documentclass[a4paper,11pt]{amsart}

\usepackage{enumerate, booktabs}
\usepackage{amsmath, amsfonts, amssymb, amsthm}
\usepackage[all]{xy}
\usepackage[paper = a4paper, left = 3.2cm, right = 3.2cm, headsep = 6mm, footskip = 10mm,
top = 32mm, bottom = 32mm, footnotesep=5mm, headheight = 2cm]{geometry}
\usepackage[active]{srcltx}
\usepackage{mathpazo}
\usepackage{hyperref}

\providecommand{\MR}{\relax\ifhmode\unskip\space\fi MR }

\providecommand{\href}[2]{#2}

\numberwithin{equation}{section}

\renewcommand{\geq}{\geqslant}

\theoremstyle{plain}
\newtheorem{thm}{Theorem}[section]

\newtheorem{lem}[thm]{Lemma}

\newtheorem*{thm*}{Theorem}

\theoremstyle{definition}

\theoremstyle{remark}
\newtheorem{rem}[thm]{Remark}
\newtheorem{ex}[thm]{Example}

\newcommand{\mbb}[1]{\mathbb{#1}}

\newcommand{\ol}[1]{\overline{#1}}
\newcommand{\lie}[1]{{\mathfrak{#1}}}
\newcommand{\abs}[1]{\lvert #1\rvert}
\newcommand{\Abs}[1]{\bigl\lvert #1\bigr\rvert}

\newcommand{\hq}{/\hspace{-0.12cm}/}

\newcommand{\acts}{\mbox{ \raisebox{0.26ex}{\tiny{$\bullet$}} }}

\DeclareMathOperator{\id}{id}

\DeclareMathOperator{\Hom}{Hom}

\DeclareMathOperator{\Lie}{Lie}

\DeclareMathOperator{\Cl}{Cl}
\DeclareMathOperator{\Pic}{Pic}

\title[Momentum maps and K\"ahler property of base spaces]{Momentum maps and 
the K\"ahler property for base spaces of reductive principal bundles}

\author{Daniel Greb}
\address{Essener Seminar f\"ur Algebraische Geometrie und Arithmetik, 
Fakult\"at f\"ur Ma\-the\-matik, Universit\"at Duisburg--Essen, 45117 Essen, 
Germany}
\email{daniel.greb@uni-due.de}
\urladdr{https://www.esaga.uni-due.de/daniel.greb}

\author{Christian Miebach}
\address{Univ.~Littoral C\^ote d'Opale, UR 2597 - LMPA - Laboratoire de 
math\'ematiques pures et appliqu\'ees Joseph Liouville, F-62100 Calais, 
France}
\email{christian.miebach@univ-littoral.fr}
\urladdr{http://www-lmpa.univ-littoral.fr/$\sim$miebach/}

\begin{document}

\begin{abstract}
We investigate the complex geometry of total spaces of reductive principal 
bundles over compact base spaces and establish a close relation between the 
K\"ahler property of the base, momentum maps for the action of a maximal compact 
subgroup on the total space, and the K\"ahler property of special equivariant 
compactifications. We provide many examples illustrating that the main result is 
optimal.
\end{abstract}

\maketitle

\section{Introduction}

Complex-reductive Lie groups $G = K^{\mathbb{C}}$ acting holomorphically on 
K\"ahlerian manifolds $X$ appear naturally in many questions of Complex 
Geometry. Classical examples, one of which found by Lescure we recall in 
Section~\ref{sect:examples} below, show that even if the action is proper and 
free with compact quotient manifold $B$, so that $X$ becomes a $G$-principal 
bundle over $B$, the K\"ahler property does in general not descend to $B$. 

In these examples, the induced action of the maximal compact subgroup $K$ of 
$G$ is not Hamiltonian with respect to any K\"ahler form on $X$, i.e., there is 
no equivariant momentum map $\mu: X \to \mathfrak{k}^* = \mathrm{Lie}(K)^*$. So, 
a natural guess might be that the K\"ahler property descends for Hamiltonian 
$G$-actions. We will see that it follows from the theory of K\"ahlerian 
Reduction, see~\cite{GS}, \cite{Sjamaar}, and~\cite{HL}, that this is true once 
in addition the zero fibre $\mu^{-1}(0)$ of the momentum map is compact, and 
conversely that principal $G$-bundles over compact K\"ahler manifolds always 
admit K\"ahler structures such that the $K$-action admits a momentum map with 
compact zero fibre. 

The group $G$ is an affine algebraic group in a unique and canonical way. As 
such, it admits a (non-unique) equivariant smooth projective completion; i.e., there 
exists a smooth projective variety $\overline{G}$ admitting an effective 
algebraic action of $G$ and containing a point $e\in \overline{G}$ that has 
trivial stabiliser in $G$ and open $G$-orbit $G \cong G \acts e \subset 
\overline{G}$. Any principal $G$-bundle $\pi\colon X \to B$ admits a fibrewise 
partial compactification to a $\overline{G}$-fibre bundle $\overline{\pi}\colon 
\overline{X} \to B$ having the same transition functions as $\pi$; using the 
associated bundle construction, this can be written as $\overline{X} = X 
\times_G \overline{G} \to X/G$. With this notation, our main result can now be 
formulated as follows: 

\begin{thm}\label{Thm:mainthm}
Let $G=K^\mbb{C}$ be a complex-reductive Lie group 
acting holomorphically, properly and freely on the connected complex manifold 
$X$  such that $Q=X/G$ is compact. Then, the following are equivalent:
\begin{enumerate}[(a)]
\item The quotient $Q$ is K\"ahler.
\item There exists a $K$-invariant K\"ahler form on $X$ with respect to which 
there is a momentum map $\mu\colon X\to \mathrm{Lie}(K)^*$ with compact zero fibre 
$\mu^{-1}(0)\not=\emptyset$.
\item The natural compactification of the principal $G$-bundle $\pi\colon X \to 
Q$ to a $\overline{G}$-fibre bundle $\overline{\pi}\colon \overline{X} \to Q$ is 
a K\"ahler manifold. 
\end{enumerate}
\end{thm}

The implication "(a) $\Rightarrow$ (b)" can be seen as a first instance of a 
Hamiltonian version of Mumford's famous GIT-statement \cite[Converse 
1.12]{MumfordGIT}, which has been vastly generalised to show that any open 
subset of a smooth quasi-projective variety admitting a projective good quotient 
is actually the set of semistable points with respect to some linearised line 
bundle, see \cite{Hausen}. 

In Section~\ref{sect:examples}, we give several examples showing that 
Theorem~\ref{Thm:mainthm} is optimal. In particular, we construct an example of 
a proper, free action of a complex-reductive group $G=K^\mathbb{C}$ on a 
K\"ahler manifold $X$ such that the $K$-action is Hamiltonian (and in fact $X$ 
admits an equivariant K\"ahler compactification on which the $K$-action is still 
Hamiltonian) with compact, non-K\"ahler 
quotient $B$; in particular, the $K$-action is Hamiltonian with respect to 
\emph{some} K\"ahler form on $X$, but there is \emph{no} K\"ahler form on $X$ 
with $K$-momentum map having non-empty compact zero fibre. 

\subsection*{Acknowledgments}

We would like to thank Karl~Oeljeklaus for having brought 
Example~\ref{Ex:Lescure} to our attention. Both authors were partially 
supported by the ANR-DFG-funded project QuaSiDy - "Quantization, Singularities, 
and Holomorphic Dynamics".  DG is also partially supported by the DFG-Research 
Training Group GRK 2553 "Symmetries and classifying spaces: analytic, 
arithmetic and derived". He wants to thank the Laboratoire de Math\'ematiques 
Pures et Appliqu\'ees Joseph Liouville at Universit\'e du Littoral C\^ote 
d'Opale for its hospitality during a visit in May 2022.

\subsection*{Notation}

Throughout, we will work over the field $\mathbb{C}$ of complex numbers. A 
\emph{variety} is a reduced scheme of finite type over $\mathbb{C}$; i.e., not 
necessarily irreducible. All complex spaces are assumed to be Hausdorff and 
second countable, so that smooth complex spaces are complex manifolds. If $G$ 
acts holomorphically on a complex manifold $X$, for any $\xi \in \mathfrak{g} = 
\mathrm{Lie}(G)$ we will denote the induced (real holomorphic) \emph{fundamental 
vector field} on $X$ by $\xi_X$; i.e., for any $f \in \mathcal{C}^\infty(X)$, we 
have
\[\xi_X(f)(p) := \left.\frac{d}{dt}\right|_{t=0} \bigl(f(\exp(t\xi)\acts 
p) \bigr).\] 
If a Lie group $K$ with Lie algebra $\mathfrak{k}:= \mathrm{Lie}(K)$ acts on a 
K\"ahler manifold $(X, \omega)$ preserving the K\"ahler form, a \emph{momentum 
map} is a $K$-equivariant map $\mu: X \to \mathfrak{k}^*$ whose components 
$\mu^\xi (\cdot) = \mu(\cdot)(\xi) \in \mathcal{C}^\infty(X)$ satisfy the
Hamiltonian equations
\[d\mu^\xi = \imath_{\xi_X}\omega.\]
If the complex structure of $X$ is denoted by $J \in 
\mathrm{End}(TX)$, slightly nonstandard, but useful for our purposes, for any $f 
\in \mathcal{C}^\infty(X)$, we set $d^c(f) := df\circ J$, so that $2i \partial 
\bar{\partial} = -d d^c$.

\section{Momentum maps for $K$-representations and their projective 
compactifications}\label{Section:Representations} 

Let $V$ be a finite-dimensional complex $G$-representation. Assume that we are 
given a $K$-invariant Hermitian product $\langle \cdot, \cdot \rangle$ on $V$, 
making the induced $K$-representation unitary. Let $\chi_V: V \to 
\mathbb{R}^{\geq 0}$ be defined by $v \mapsto \langle v,v \rangle$ and consider 
the K\"ahler form $-dd^c\chi_V$ on $V$. Note that the good 
quotient $\pi_V: V \to V\hq G = \mathrm{Spec}\left(\mathbb{C}[V]^G \right)$ 
exists; it is endowed with a K\"ahler structure by K\"ahlerian reduction with 
respect to the momentum map \[\mu_V: V \to \mathfrak{k}^*, \; v \mapsto (\xi 
\mapsto 2\langle i\xi\cdot v, v\rangle) =  d^c \chi_V (\xi_V)(v). \] These 
however are not the right K\"ahler form and momentum map to consider in our 
situation, since they do not globalise to the situation of a $G$-vector bundle 
over a non-trivial base manifold with typical fibre $V$, 
see~Example~\ref{Ex:WrongBundleMetric}.

Instead, we are going to compactify the situation. For this, we consider the 
$G$-representation space $W:=V\oplus\mbb{C}$, where $G$ acts trivially on 
$\mbb{C}$, and embed $V$ regularly and equivariantly into $\mathbb{P}(W) 
=\mathbb{P}(V \oplus \mathbb{C})$; explicitly we consider \[\theta: V 
\hookrightarrow \mathbb{P}(V \oplus \mathbb{C}),\; v \mapsto [v:1].\] 
We denote by $\mathbb{P}(W)^{ss}$ the Zariski-open subset of 
points that are semistable with respect to the linearisation of the $G$-action 
on $\mathbb{P}(W)$ in $W$ or equivalently the corresponding 
linearisation in the line bundle $\mathcal{O}_{\mathbb{P}(W)}(1)$; the associated good quotient 
will be called $\pi_{\mathbb{P}(W)}\colon\mathbb{P}(W)^{ss} 
\to \mathbb{P}(W)^{ss}\hq G$. While not strictly needed for 
our subsequent arguments, the following observation regarding GIT is at least 
philosophically crucial.

\begin{lem}\label{lem:saturated}
The representation space $V$ is mapped isomorphically to a 
$\pi_{\mathbb{P}(W)}$-saturated subset of $\mathbb{P}(W)^{ss}$ 
via the $G$-equivariant map $\theta$.
\end{lem}

\begin{proof}
This is almost tautological: if we choose linear coordinates (i.e., linear 
forms) $z_1, \dots, z_{\dim V}, z_0$ on $W=V \oplus \mathbb{C}$, then $z_0$ is 
$G$-invariant, as the action of $G$ on $\mathbb{C}$ is trivial. This implies 
that the associated section $\sigma = \sigma_{z_0} \in H^0\bigl(\mathbb{P}(W),  
\mathcal{O}_{\mathbb{P}(W)}(1)\bigr)$ is $G$-invariant. Since obviously 
\[\theta (V) = \{[w] \in \mathbb{P}(W) \mid \sigma (w) \neq 0\} 
= \mathbb{P}(W)_{\sigma},\] every point in $\theta(V)$ is 
semistable. Moreover, for any $\tau \in H^0\bigl(\mathbb{P}(W), 
\mathcal{O}_{\mathbb{P}(W)}(1)\bigr)^G$, the corresponding open subset  
$\mathbb{P}(W)_\tau$ is $\pi_{\mathbb{P}(W)}$-saturated; see the proof of 
\cite[Thm.~1.10]{MumfordGIT}
\end{proof}

Consequently, we will consider $V \subset \mathbb{P}(V \oplus \mathbb{C})$, 
suppressing $\theta$. Using the linear coordinate $z_0$ on $\mathbb{C}$ 
introduced in the previous proof, we endow the $G$-representation $W= V \oplus 
\mathbb{C}$ with the Hermitian form $\langle \cdot , \cdot \rangle$  
corresponding to \[\chi_W (v,z_0) := \chi_V(v) + |z_0|^2.\] This in turn yields 
a Fubini-Study form on $\mathbb{P}(W)$, which is $K$-invariant and makes the 
$K$-action Hamiltonian with momentum map given by 
\[\mu^\xi ([w]) = \frac{2 \langle i\xi\cdot w, w \rangle}{\langle w, w \rangle}.\]

The next observation is a momentum geometry counterpart of 
Lemma~\ref{lem:saturated} and will be applied fibrewise in the bundle situation 
considered in Section~\ref{sect:bundlecase} below. 

\begin{lem}\label{lem:FSpotential}
A potential for the restriction of the Fubini-Study K\"ahler form $\omega = 
\omega_{\mathbb{P}(W)}$ to $V \subset \mathbb{P}(W)$ is given  
by
\[\omega|_V = - d d^c \log (\chi_V  +1 ) = 2i \partial \bar\partial 
\log (\chi_V  +1 ).\]
The function $\rho := \log (\chi_V  +1)  \in \mathcal{C}^\infty (V)$ is 
an exhaustion of $V$, and if $\xi_{\mathbb{P}(W)}$ is the vector field on 
$\mathbb{P}(W)$ induced by the $K$-action, then 
\[ V \ni v \mapsto d^c\rho (\xi_{\mathbb{P}(W)})([v:1])  =   \mu^\xi([v:1]) \]
defines a momentum map for the $K$-action on $V$  with respect to $\omega|_V$. 
\end{lem}

\begin{proof}
The first part is well-known, the exhaustion property is clear, and the last 
part follows by direct computation that uses the obvious equality
\[\xi_W(v,1) = (\xi_V(v), 0)\] together with the fact that (being a 
$G$-equivariant isomorphism) $\theta^{-1}$ transforms the fundamental vector 
field $\xi_{\mathbb{P}(W)}$ into $\xi_V$ .
\end{proof}

\begin{rem}\label{rem:potentialmu}
In the following, we will only need that $\rho$ is an exhaustive $K$-invariant 
potential for a K\"ahler form on $V$. The fact that $d^c\rho$ then yields a 
momentum map for the $K$-action can be seen as follows: since $\rho$ and hence 
$d^c\rho$ is $K$-invariant, we have $\mathcal{L}_{\xi_X}(d^c\rho) = 0$, and 
therefore Cartan's formula yields \[d\bigl(\imath_{\xi_X} (d^c\rho) \bigr) = 
-\imath_{\xi_X}\bigl(dd^c(\rho)\bigr) = \imath_{\xi_X}\omega, \]
so that indeed $\imath_{\xi_X} (d^c\rho)$ is the $\xi$-component of a momentum 
map with respect to $\omega$, whose equivariance is easy to check as well.
\end{rem}

\section{K\"ahler forms on holomorphic principal bundles and vector 
bundles}\label{sect:bundlecase}

Given a $G$-principal bundle $\pi\colon X \to Q$ over a compact K\"ahlerian 
manifold $Q = X/G$, we will be looking for special K\"ahler structures on $X$. 
For this, the following construction will be useful.

\begin{rem}\label{rem:vectorbundleembedding}
As $G$ is an affine algebraic $G$-variety (with algebraic action given by 
left-multiplication), it admits a closed (algebraic) embedding $\psi: G 
\hookrightarrow V$ into the vector space associated with a certain 
finite-dimensional $G$-representation $\varphi: G \to GL_\mathbb{C}(V)$, see for 
example \cite[Prop.~2.2.5]{Brion}. I.e., there is a closed $G$-orbit $G \acts v$ 
in $V$ with $G_v =\{e\}$, and in particular, $\varphi$ is injective. We may 
assume that $V$ has a $K$-invariant Hermitian form $h$, so that 
$\varphi|_K\colon K \hookrightarrow SU(h)$. The associated vector bundle 
$\mathcal{V}:= X \times_G V$ has a natural holomorphic $G$-action (preserving 
each fibre), as well as a Hermitian metric that is induced by $h$ and therefore 
invariant by the associated $K$-action. Slightly abusing notation, we will call 
the bundle metric $h$ as well. Writing $X = X \times_G G$, we may produce a 
closed $G$-equivariant holomorphic embedding 
 \[\begin{xymatrix}{
    X=X\times_G G\, \ar@{^(->}^{\Psi:= \id_X\times_G \psi}[rr] \ar[rd]& & 
X\times_G 
    V =\mathcal{V} \ar[ld]\\
      &X/G &
    }
   \end{xymatrix}
\]
of $X$ into $\mathcal{V}$ over $Q = X/G$.
\end{rem}

Let now $(B,\omega)$ be a compact K\"ahler manifold and let $\pi\colon E\to B$ 
be a holomorphic vector bundle. 

\begin{rem}\label{rem:Voisin}
Setting $\mathcal{W} := \mathcal{V} \oplus \mathbb{C}$, it can be deduced 
from~\cite[Th\'eor\`eme Principal II]{Bl} that the total space of 
$\mbb{P}(\mathcal{W})\to B$ is K\"ahler. In fact, there is a rather explicit 
way of constructing K\"ahler forms on projectivisations of vector bundles, 
e.g.~see~\cite[Prop.~3.18]{V1}. As we will see in more detail below, one uses 
the fact that there exists a $(1,1)$-form $\omega_\mathcal{W}$ on 
$\mbb{P}(\mathcal{W})$ whose restriction to every fibre 
$\mbb{P}(\mathcal{W}_x)$ is the Fubini-Study form of this fibre associated with 
a bundle metric on $\mathcal{W}$, namely (up to some positive constant) the 
Chern form of $\mathcal{O}_{\mbb{P}(\mathcal{W})}(1)$. Then one obtains a 
K\"ahler form on $\mbb{P}(\mathcal{W})$ by adding a sufficiently large multiple 
of the pull-back of $\omega$ to $\mbb{P}(\mathcal{W})$. Since the holomorphic 
vector bundle $\mathcal{V}$ embeds as an open subset into 
$\mbb{P}(\mathcal{V}\oplus\mbb{C})$, we see that $\mathcal{V}$ and hence $X 
\subset \mathcal{V}$ both inherit corresponding K\"ahler forms by restriction. 
\end{rem}

We will now analyse the construction of the K\"ahler form on $\mathcal{V} 
\subset \mathbb{P}(\mathcal{V} \oplus \mathbb{C})$ in more detail, with the aim 
of pointing out the specific features relevant to momentum map geometry. 

\begin{lem}\label{lem:VectorBundle}
Let $\pi\colon \mathcal{V}\to B$ be a holomorphic vector bundle of rank $n$ 
over a compact K\"ahler manifold $(B,\omega)$ on which the compact Lie group 
$K$ acts holomorphically via fibre-preserving vector bundle automorphisms. Set 
$\mathfrak{k}:=\Lie (K)$. Let  $h$ be a $K$-invariant Hermitian metric $h$ on 
$\mathcal{V}$, with associated length function 
\[\chi_h: \mathcal{V} \to \mathbb{R}^{\geq 0}, \; v \mapsto h(v,v).\]
Then, for all positive real numbers $c \gg 0$ the $(1,1)$-form 
\[\omega_\mathcal{V} = 2i \partial \bar\partial \log(\chi_h  + 1)  + c \cdot 
\pi^*(\omega) \]
has the following properties:
\begin{enumerate}
\item $\omega_\mathcal{V}$ is K\"ahler, 
\item the $K$-action on $\mathcal{V}$ admits a momentum map $\mu\colon 
\mathcal{V} \to \mathfrak{k}^*$ with respect to $\omega_\mathcal{V}$,
 \item every point $x \in B$ has an open neighborhood $U$ such that on 
$\pi^{-1}(U)$ we may write $\omega = 2i \partial\bar\partial \rho$ for some 
function $\rho \in \mathcal{C}^\infty(\pi^{-1}(U))^{K}$ that has the following 
properties:
 \begin{enumerate} 
 \item $\rho$ is an exhaustion when restricted to  $\mathcal{V}_y$ for all $y 
\in U$, 
 \item for all $\xi \in \mathfrak{k}$ and for all $v \in \pi^{-1}(U)$ we have
 \begin{equation}\label{eq:momentumfromrho}
  \mu^\xi (v) = d^c\rho(\xi_\mathcal{V}(v)) = d\rho(J\xi_\mathcal{V}(v)), 
 \end{equation}
where as usual $J \in \mathrm{End}(T\mathcal{V})$ denotes the (almost) complex structure 
of the complex manifold $\mathcal{V}$, and $\xi_\mathcal{V}$ is the vector field 
on $\mathcal{V}$ associated with $\xi$ via the $K$-action. 
\end{enumerate}
\end{enumerate}
\end{lem}

\begin{proof}
We consider $\mathcal{W} = \mathcal{V} \oplus \mathbb{C}$, the direct sum of 
$\mathcal{V}$ with the trivial line bundle, and its projectivisation 
$\mathbb{P}(\mathcal{W}) = (\mathcal{W} \setminus 
z_{\mathcal{W}})/\mathbb{C}^*$, which contains $\mathbb{P}(\mathcal{V})$ as a 
codimension one subbundle with (open) complement isomorphic to $\mathcal{V}$. 
Endow $\mathcal{W}$ with the Hermitian metric \[h_\mathcal{W}(v \oplus z) := 
\chi_h(v) + |z|^2.\] Using the section of 
$\mathcal{O}_{\mathbb{P}(\mathcal{W})}(1)$ corresponding to the divisor 
$\mathbb{P}(\mathcal{V}) \subset \mathbb{P}(\mathcal{W})$ one computes that the 
\emph{relative Fubini-Study form} $\omega_{FS}^{\mathcal{W}} $, i.e., the 
curvature form of the natural Hermitian metric in 
$\mathcal{O}_{\mathbb{P}(\mathcal{W})}(1)$ induced by $h$, when restricted to 
$\mathcal{V} = \mathbb{P}(\mathcal{W}) \setminus \mathbb{P}(\mathcal{V})$  is 
given by 
\[c\cdot \omega_{FS}^{\mathcal{W}}|_{\mathcal{V}} = 2i\partial \bar \partial 
\log(\chi_h +1) = - d d^c \log(\chi_h + 1)\quad \quad \text{for some 
}c\in \mathbb{R}^{>0};\]
see for example \cite[Chap.~V, \S 15.C; p.~282]{Dem}, but notice the different 
convention regarding projectivisations of vector bundles. Restricting to any 
fibre $\mathcal{V}_y$, $y \in B$, we recover the K\"ahler form associated with 
the Hermitian scalar product $h_y$ discussed in Lemma~\ref{lem:FSpotential}. 
Using compactness of $B$, Part (1) is now proven with the argument of 
\cite[p.~78]{V1}. 

The statements made in Parts (2) and (3) will be proven simultaneously. Let 
$\{U_k\}_{k \in I}$ be a finite open covering of $B$ so that $\omega|_{U_k}$ is 
given by a potential $\varphi_k$; i.e., $\omega|_{U_k} = - d d^c \varphi_k$.  It is clear that 
$\omega_\mathcal{V}|_{\pi^{-1}(U_k)}$ has $K$-invariant potential 
\begin{equation}\label{eq:localpotential}\rho_k := \log(\chi_h + 
1)|_{\pi^{-1}(U_k)} + c\cdot\pi^*(\varphi_k).
\end{equation} 
On any of the open sets $\pi^{-1}(U_k)$ a corresponding momentum map for the $K$-action 
is given by $\mu^\xi_k:=\iota_{\xi_\mathcal{V}}d^c\rho_k$, 
cf.~Remark~\ref{rem:potentialmu}. Setting $\eta:= \log(\chi_h + 1)$ in order to 
shorten notation, still on $\pi^{-1}(U_k)$ we compute
\begin{align}
\iota_{\xi_\mathcal{V}}d^c\bigl(\eta+c\cdot\pi^*\varphi_k\bigr)&=\iota_{
\xi_\mathcal{V}}d^c\eta + c\cdot 
\iota_{\xi_\mathcal{V}}d^c\pi^*\varphi_k  \nonumber\\
&=\iota_{\xi_\mathcal{V}}d^c\eta + 
c\cdot\iota_{\xi_\mathcal{V}}\pi^*d^c\varphi_k \label{eq:computation}\\&    
=\iota_{\xi_\mathcal{V}}d^c\eta \nonumber
\end{align}
Here, we used the fact that $\xi_\mathcal{V}$ is tangential to the fibres of 
$\pi$, whereas the forms $\pi^*d^c\varphi_k$ vanish in fibre direction. 
Consequently, the maps $\mu_k$ on $\pi^{-1}(U_k)$ glue together to a 
well-defined momentum map $\mu_\mathcal{V}\colon \mathcal{V}\to\lie{k}^*$ with 
respect to $\omega_\mathcal{V}$; as the computation \eqref{eq:computation} 
shows, it is given by the map formally associated with the (in general only 
fibrewise strictly plurisubharmonic) function $\eta = \log (\chi_h +1 )$ defined 
on the whole of $\mathcal{V}$. This shows the statements made in Part (2). The 
properties listed in Part (3) then follow from the definition of $\rho_k$, see 
Equation~\eqref{eq:localpotential}, and the fibrewise properties established in 
Lemma~\ref{lem:FSpotential}.
\end{proof}

The following example explains the choice of the potential $\log(\chi_h+1)$ and 
the remark at the end of the first paragraph of 
Section~\ref{Section:Representations}.

\begin{ex}\label{Ex:WrongBundleMetric}
The total space of the line bundle $\mathcal{O}(1)$ over $\mbb{P}_1$ can be 
explicitly realised as $\mbb{P}_2\setminus\bigl\{[0:0:1]\bigr\}$ with bundle 
projection $\pi\bigl([z_0:z_1:z_2]\bigr)=[z_0:z_1]$. Let us consider the bundle 
metric on $\mathcal{O}(1)$ with associated length function
\begin{equation*}
\chi_h\bigl([z_0:z_1:z_2]\bigr)=\frac{\abs{z_2}^2}{\abs{z_0}^2+ 
\abs{z_1}^2}.
\end{equation*}
An explicit computation shows that there does not exist any $c>0$ such that 
$i\partial\ol{\partial}\chi_h+c\cdot\pi^*\omega_{FS}$ is positive. Namely, in 
the affine coordinates on 
$\mbb{P}_2\setminus\bigl\{[0:0:1]\bigr\}\cap\{z_0\not=0\}\cong 
\mbb{C}^2\setminus\{0\}$ the Hermitian form of 
$i\partial\ol{\partial}\chi_h+c\cdot\pi^*\omega_{FS}$ is given by
\begin{equation*}
\frac{1}{\bigl(\abs{z}^2+\abs{w}^2\bigr)^3}
\begin{pmatrix}
\abs{z}^2\bigl(1+c\abs{w}^2\bigr)-\abs{w}^2+c\abs{w}^4 & 
\bigl(2-c\bigl(\abs{z}^2+\abs{w}^2\bigr)\bigr)z\ol{w}\\
\bigl(2-c\bigl(\abs{z}^2+\abs{w}^2\bigr)\bigr)\ol{z}w &
\bigl(1+c\abs{z}^2\bigr)\abs{w}^2-\abs{z}^2+c\abs{z}^4
\end{pmatrix},
\end{equation*}
which is never positive definite.
\end{ex}

\section{Proof of Theorem~\ref{Thm:mainthm}}

Before we start the proof proper, we observe that in case that $G$ is not 
connected with identity component $G^\circ$, the quotient $Q^\circ := X / 
G^\circ$ is a finite topological covering of $X/G$ and therefore also compact. 
Moreover, for $l \in \{a,b,c\}$ statement ($l$) of Theorem~\ref{Thm:mainthm} 
holds if and only if the corresponding statement ($l^\circ$) holds for the 
action of $G^\circ$, the maximal compact group $K^\circ := K \cap G^\circ$ of 
$G^\circ$, the equivariant compactification $\overline{G^\circ} \subset 
\overline{G}$, and the quotient $Q^\circ$. For this observation, the main point 
to notice is that $\overline{G} = G \times_{G^\circ} \overline{G^\circ}$, and 
hence
\[\overline{X} = X \times_G \overline{G} = X \times_G G \times_{G^\circ} 
\overline{G^\circ} \cong X \times_{G^\circ} \overline{G^\circ}.\]
We therefore assume for the remainder of the proof that $G=G^\circ$ is 
connected.

\subsection{Implication "(b) $\Rightarrow$ (a)"}

We quickly summarise the fundamental results in the K\"ahlerian version of the 
theory of Symplectic Reduction: Let $G=K^\mbb{C}$ be a complex-reductive group 
and let $X$ be a holomorphic $G$-manifold. Let $\omega$ be a $K$-invariant 
K\"ahler form on $X$ such that the $K$-action on $X$ is Hamiltonian with 
equivariant momentum map $\mu\colon X\to\lie{k}^*$. If the zero fibre 
$\mu^{-1}(0)$ is non-empty, then the set 
\begin{equation*}
X^{ss}(\mu):=\bigl\{x\in X \mid \ol{G\acts x}\cap\mu^{-1}(0)\not=\emptyset\bigr\}
\end{equation*}
of semistable points is a non-empty open $G$-invariant of $X$ admitting an 
\emph{analytic Hilbert quotient} $q\colon X^{ss}(\mu)\to X^{ss}(\mu)\hq G$; 
i.e., $q$ is a surjective $G$-invariant, (locally) Stein map to a complex space 
$X^{ss}(\mu)\hq G$ fulfilling \[q_*(\mathcal{O}_{X^{ss}(\mu)})^G = 
\mathcal{O}_{X^{ss}(\mu)\hq G}.\] In particular, $q$ is universal with respect 
to $G$-invariant holomorpic maps from $X^{ss}(\mu)$ to complex spaces. Moreover, 
every fibre of $q$ contains a unique $G$-orbit that is closed in $X^{ss}(\mu)$; 
this orbit is the only one in $\pi^{-1}(0)$ that intersects $\mu^{-1}(0)$, and 
the intersection is a uniquely determined $K$-orbit. With more work, one can 
show that $X^{ss}(\mu)\hq G$ is homeomorphic to the symplectic reduction 
$\mu^{-1}(0)/K$ and via this homeomorphism inherits a K\"ahler structure induced 
by $\omega$, see \cite{HL, HHL, Sjamaar}, and the detailed 
survey~\cite{HHSurvey}.

Suppose from now on that in addition to the assumptions made above the 
$G$-action on $X$ is free and proper.  In this situation, since all $G$-orbits 
are closed in $X$, we have \[X^{ss}(\mu)=G\cdot\mu^{-1}(0),\] and hence 
$X^{ss}(\mu)\hq G=X^{ss}(\mu)/G =\pi(X^{ss}(\mu))$ is an open subset of $X/G$ 
that is homeomorphic to $\mu^{-1}(0)/K$ and possesses a K\"ahler form, whose 
construction is much easier in this simple case; e.g., see \cite[Thm.~3.1]{HKLR} 
and cf.~\cite[Thm.~3.5]{GS} or \cite[Prop.~2.4.6]{HHSurvey}. Under the 
assumption made in (b), the non-empty subset $\pi(X^{ss}(\mu)) \simeq 
\mu^{-1}(0)/K$ is not only open, but also compact. It therefore coincides with 
$X/G=Q$, which  is hence K\"ahler. 

\subsection{Implication "(a) $\Rightarrow$ (b)"}

Choose a $G$-equivariant closed holomorphic embedding $\Psi\colon X 
\hookrightarrow \mathcal{V}$ into the total space of a holomorphic $G$-vector 
bundle over $Q$, as constructed in Remark~\ref{rem:vectorbundleembedding}, and 
let $h$ be a corresponding $K$-invariant Hermitian metric on $\mathcal{V}$. 
Next, we apply Lemma~\ref{lem:VectorBundle} to obtain a constant $c \gg 0$ such 
that \[\omega = 2i \partial \bar\partial \log(\chi_h + 1) + c \cdot 
\pi^*(\omega)\] is a K\"ahler form with the properties (2) and (3) stated in the 
Lemma. We will show that the restriction of $\omega$ to the closed submanifold 
$X \subset \mathcal{V}$, which we continue to denote by $\omega$, has the 
properties claimed in Theorem~\ref{Thm:mainthm}(b). 

First, we notice that the $K$-action on $X$ is Hamiltonian with respect to 
$\omega$, with momentum map being given by the restriction of the momentum map 
from $\mathcal{V}$ to $X$, which we continue to denote by $\mu$. Now, let $q\in 
Q$ be any point, let $x \in X$ a point in the fibre over $q$, and let $U$ and 
$\rho$ be as in part (3) of the Lemma. Then, by (3.a) the restriction of $\rho$ 
to $X_q = G\acts x \subset \mathcal{V}_q$ continues to be a strictly plurisubharmonic 
exhaustion function. Therefore, $\rho|_{X_q}$ has a minimum and in particular a 
critical value, say at $x_0 \in G\acts x$. As $J\xi_\mathcal{V}(x_0) = J\xi_{X_q}(x_0) = 
J\xi_{G\acts x_0} (x_0)$ is obviously tangent to the (complex) orbit $G\acts x_0 
= G\acts x$, Formula \eqref{eq:momentumfromrho} together with the fact that $x_0$ is critical implies that 
\[\mu^\xi(x_0) = d\rho(J \xi_{\mathcal{V}}(x_0)) = d\rho(J \xi_{X_q}(x_0)) = 0 \quad 
\quad \text{for all }\xi \in \mathfrak{k},\]
so that $x_0 \in \mu^{-1}(0)$. In other words, every fibre of $\pi\colon X \to Q$ intersects $\mu^{-1}(0)$, so 
that  \[\pi|_{\mu^{-1}(0)}\colon \mu^{-1}(0) \to Q\] is a $K$-principal bundle 
over the compact manifold $Q$. In particular, $\mu^{-1}(0) \subset X$ is 
compact. 

\subsection{Implication "(a) $\Rightarrow$ (c)"}

As a normal projective $G$-variety for the connected group $G$, the completion 
$\overline{G}$ admits a $G$-equivariant embedding into the projectivisation 
$\mathbb{P}(W)$ of a finite-dimensional complex $G$-representation $W$; see 
\cite[Prop.~5.2.1, Prop.~3.2.6]{Brion}. Using the associated bundle 
construction, this leads to a closed holomorphic embedding
\[\begin{xymatrix}{
\overline{X} = X \times_G \overline{G}\, \ar[rd]\ar@{^(->}[r] & X \times_G 
\mathbb{P}(W) \ar[d] \ar@{=}[r]& \mathbb{P}(\mathcal{W}) \ar[ld]\\
 & Q & 
}
  \end{xymatrix}
   \]
into the projectivisation of the holomorphic vector bundle $\mathcal{W} = X 
\times_G W$ over $Q$. By assumption, $Q$ is compact K\"ahler, so 
Remark~\ref{rem:Voisin} yields the claim.  

\subsection{Implication "(c) $\Rightarrow$ (a)"}

Since $\overline{\pi}\colon \overline{X} \to Q$ is a fibre bundle with compact
K\"ahlerian to\-tal space (and connected fibres), the claim follows from 
\cite[Prop.~II.2]{Bl}.\newline

This concludes the proof of Theorem~\ref{Thm:mainthm}. \hfill\qed

\section{Examples}\label{sect:examples}

The following example shows that even the "K\"ahler" statement of the 
implication $(a)\Longrightarrow(b)$ of Theorem~\ref{Thm:mainthm} does not hold 
for \emph{non-compact} $Q$.

\begin{ex}\label{Ex:NonKählerPBwithNonCompactKählerBase}
The homogeneous fibration
\begin{equation*}
{\rm{SL}}(2,\mbb{C})/
\left(\begin{smallmatrix}1&\mbb{Z}\\0&1\end{smallmatrix}\right)\to
{\rm{SL}}(2,\mbb{C})/
\left(\begin{smallmatrix}1&\mbb{C}\\0&1\end{smallmatrix}\right)
\end{equation*}
is a $\mbb{C}^*$-principal bundle over the non-compact K\"ahler base 
$\mbb{C}^2\setminus\{0\}$. Owing to~\cite[Theorem~3.1]{BO} however, the total 
space ${\rm{SL}}(2,\mbb{C})/\left(\begin{smallmatrix}1&\mbb{Z}\\0&1 
\end{smallmatrix}\right)$ is \emph{not} K\"ahler.
\end{ex}

A slight modification of Example~\ref{Ex:NonKählerPBwithNonCompactKählerBase} 
yields an example showing that an analogue of Lemma~\ref{lem:VectorBundle} does 
not hold if we replace the vector bundle $\pi\colon E\to X$ by an arbitrary  
holomorphic fibre bundle having Stein fibres, even though these kind of fibres 
also admit strictly plurisubharmonic exhaution functions similar to (the logarithm 
of) the norm square function used in the proof of Lemma~\ref{lem:VectorBundle}.

\begin{ex}\label{Ex:SteinBundle}
Let $B$ be the Borel subgroup of $G={\rm{SL}}(2,\mbb{C})$ consisting of upper 
triangular matrices and let $\Gamma=\left(\begin{smallmatrix}1&\mbb{Z}\\0&1 
\end{smallmatrix}\right)$. One verifies directly that $B/\Gamma$ is 
biholomorphic to $(\mbb{C}^*)^2$ and therefore a Stein manifold. Consequently, 
the homogeneous fibration $G/\Gamma\to G/B$ is a holomorphic fibre bundle over 
the compact K\"ahler manifold $\mbb{P}_1$ with typical fibre $(\mbb{C}^*)^2$ 
such that the total space $G/\Gamma$ is \emph{not} K\"ahler, again 
by \cite[Theorem~3.1]{BO}. Note that this bundle is \emph{not} principal, as 
$\Gamma$ is not normal in $B$.
\end{ex}

Next, we give an example (originally due to Lescure) which explains that even 
for $K^\mathbb{C}$-actions on \emph{K\"ahlerian} manifolds $X$ further 
conditions are needed to ensure that the quotient $X/G$ is K\"ahler.

\begin{ex}\label{Ex:Lescure}
Let us consider the Hopf surface $Y=(\mbb{C}^2\setminus\{0\})/\mbb{Z}$ with 
respect to the $\mbb{Z}$-action given by $m\acts v=2^mv$. Let $p\colon 
\mbb{C}^2\setminus\{0\}\to Y$ be the universal covering. The group 
$G={\rm{GL}}(2, \mbb{C})$ acts transitively on $Y$, and a direct calculation 
shows that the isotropy group of $y_0:=p(1,0)\in Y$ is
\begin{equation*}
G_{y_0}=\left\{
\begin{pmatrix}
2^m&b\\0&a
\end{pmatrix};\ m\in\mbb{Z}, b\in\mbb{C}, a\in\mbb{C}^*\right\}.
\end{equation*}
Let us consider the closed subgroups
\begin{equation*}
H:=\left\{
\begin{pmatrix}
2^m&b\\0&1
\end{pmatrix};\ m\in\mbb{Z}, b\in\mbb{C}\right\}\text{ and }
T:=\left\{
\begin{pmatrix}
1&0\\0&a
\end{pmatrix};\ a\in\mbb{C}^*\right\}
\end{equation*}
of $G_{y_0}$. One verifies directly that $T\cong\mbb{C}^*$ normalises $H$ and thus that 
$G_{y_0}=TH\cong T\ltimes H$. Hence, we have the holomorphic 
$\mbb{C}^*$-principal bundle $X:=G/H\to Y=G/(HT)$.

We claim that $X$ is a K\"ahler manifold. For this, let 
$S:={\rm{SL}}(2,\mbb{C})$ and note that
\begin{enumerate}[(1)]
\item $S\cap H=\left(\begin{smallmatrix}1&\mbb{C}\\0&1\end{smallmatrix}\right)$ 
is an algebraic subgroup of $S$, as well as that
\item $SH=\bigl\{2^{m/2}g;\ m\in\mbb{Z}, g\in S\bigr\}$ is a closed subgroup of 
$G$. (In fact, $G/(SH)$ is an elliptic curve.)
\end{enumerate}
Now we can apply~\cite[Theorem~5.1]{GMO} to deduce that $X$ is K\"ahler. 

However, there is no momentum map for the $S^1$-action on $X=G/H$ with respect 
to any $S^1$-invariant K\"ahler form: Indeed, since the subgroup $H$ has 
infinitely many connected components and is therefore not algebraic, the action 
of ${\rm{U}}(2)$ on $X$ is not Hamiltonian, see~\cite[Theorem~4.11]{GMO}. Since 
${\rm{U}}(2)={\rm{SU}}(2)S^1$, this implies that $X$ is not a Hamiltonian 
$S^1$-manifold, either.
\end{ex}

In Example~\ref{Ex:Lescure} above, while $X$ is K\"ahler, the action of a 
maximal compact group $K$ is not Hamiltonian with respect to any K\"ahler form. 
Therefore, one might wonder whether the existence of a momentum map for the 
$K$-action is actually sufficient to conclude that the base of a K\"ahlerian 
principal bundle is K\"ahler. Candidates of counterexamples to this statement 
can be obtained by the Cox construction, which realises toric varieties 
associated with simplicial fans and without torus factors as geometric quotients 
of certain domains in $\mbb{C}^N$ by a linear action of a complex torus $T$, 
see~\cite[Theorem~5.1.11]{CLS}. The following concrete example, which is in some sense minimal, consists of
a smooth complete non-projective toric threefold for which Cox' geometric 
quotient is a $T$-principal bundle.

\begin{ex}\label{ex:toric}
Let $\Sigma$ be the simplicial fan from~\cite[Example~2]{FP}. It is shown that 
the associated toric threefold $X_\Sigma$ is smooth, complete and 
non-projective, hence non-K\"ahler by Moishezon's Theorem, \cite[Chap.~VII, 
Thm.~6.23]{SCVVII}. Since $\Abs{\Sigma(1)}=8$ and $X_\Sigma$ is smooth, we have
\begin{equation*}
\Cl(X_\Sigma)=\Pic(X_\Sigma)\cong\mbb{Z}^{\abs{\Sigma(1)}-\dim X(\Sigma)} 
=\mbb{Z}^5,
\end{equation*}
see~\cite[Theorem~4.1.3 and Proposition~4.2.6]{CLS}. Moreover, since $X_\Sigma$ 
has no torus factors, the relevant group is $T=\Hom_\mbb{Z}(\mbb{Z}^5,\mbb{C}^*) 
\cong(\mbb{C}^*)^5$, see~\cite[Lemma~5.1.1]{CLS}. Consequently, an application 
of~\cite[Theorem~5.1.11]{CLS} implies that there exists a geometric $T$-quotient 
$\pi\colon\mbb{C}^8\setminus Z(\Sigma)\to X_\Sigma$ where $T$ acts linearly on 
$\mbb{C}^8$. In particular, there is a momentum map for the action of any 
compact real form of $T$ on $\mbb{C}^8\setminus Z(\Sigma)$, described explicity 
in Section~\ref{Section:Representations} above.

Using again the fact that $X_\Sigma$ is smooth, we may 
apply~\cite[Proposition~2.1.4.6]{ADHL} to see that $T$ acts \emph{freely} on 
$\mbb{C}^8\setminus Z(\Sigma)$. Since the orbit space $\bigl(\mbb{C}^8\setminus 
Z(\Sigma)\bigr)/T\cong X_\Sigma$ is Hausdorff and since the holomorphic slice 
theorem for Hamiltonian actions, e.g.~see \cite[Thm.~4.1.]{HHSurvey}, yields 
local slices for the $T$-action, it follows that $T$ acts \emph{properly} on 
$\mbb{C}^8\setminus Z(\Sigma)$, see~\cite[Theorem~1.2.9]{Palais} ; i.e., $\pi$ 
is indeed a $T$-principal bundle. 
\end{ex}

\begin{rem}
As in Section~\ref{Section:Representations} above, Example~\ref{ex:toric} can 
be equivariantly compactified $\mbb{C}^8 \hookrightarrow \mathbb{P}^8 = 
\mathbb{P}(\mbb{C}^8 \oplus \mathbb{C})$ to a Hamiltonian action of $K = 
(S^1)^5$ on $\mathbb{P}^8$, say with momentum map $\mu$.  Now, given any point 
$p$ in the open subset $\mbb{C}^8\setminus Z(\Sigma)$, we may actually 
shift the momentum map by a constant $\xi \in \mathfrak{k}^*$ to define a new 
momentum map $\hat \mu$ with $p \in \hat\mu^{-1}(0)$. The corresponding set of 
 semistable points is Zariski-open in $\mathbb{P}^8$, hence its intersection 
with $\mbb{C}^8\setminus Z(\Sigma)$ yields a Zariski-open subset $U \subset 
(\mbb{C}^8\setminus Z(\Sigma)) /T \cong X_\Sigma$  admitting a K\"ahler form 
$\omega_U$. As all quotient maps are in fact meromorphic, $X_\Sigma$ is 
bimeromorphic to the compact K\"ahler space $\hat{\mu}^{-1}(0)/K$ on which 
$\omega_U$ extends to an honest K\"ahler structure by K\"ahlerian reduction; 
cf.~\cite{Fujiki}. I.e., the K\"ahlerian reduction theory  explains in a precise 
way how $X_\Sigma$ is in Fujiki's class $\mathcal{C}$ (and not K\"ahler).
\end{rem}

\begin{rem}
With respect to item (c) of Theorem~\ref{Thm:mainthm}, Example~\ref{ex:toric} 
shows that (for connected $G$) it is not enough to assume 
 \begin{enumerate}
  \item some compactification $\widehat{X}$ of $X$ to be K\"ahler, and 
  \item the action map $G \times X \to X$ to extend  to a meromorphic map 
$\overline{G} \times \widehat{X} \dasharrow \widehat{X}$
 \end{enumerate}
in order for $X/G$ to be K\"ahler. While under these two assumptions there will 
always be a momentum map $\hat{\mu}\colon \widehat{X} \to \mathfrak{k}^*$ for 
the $K$-action, see e.g.~the discussion of classical results regarding this 
connection in \cite[Rem.~2.2]{GM18}, as in the example the intersection of the 
compact subset $\widehat{\mu}^{-1}(0)$ with the open subset $X \subset 
\widehat{X}$ might always be non-compact (or empty). 
\end{rem}

\vfill

\end{document}